\documentclass{BICA}

\usepackage{tikz}
\usepackage{float}
\usepackage{xcolor}
\usepackage{aliascnt}
\usetikzlibrary{arrows}
\usetikzlibrary{decorations.pathreplacing}
\usepackage[ruled]{algorithm2e}
\usepackage{multicol}

\usepackage{cleveref}

\newcommand{\rank}{\operatorname{rank}}
\newcommand{\mnull}{\operatorname{null}}
\newcommand{\cycles}{\operatorname{cyc}}
\newcommand{\Forest}{\operatorname{For}}
\newcommand{\Cycles}{\operatorname{Cyc}}

\renewcommand{\qedsymbol}{\rule{1ex}{1ex}}

\title{2-switch: transition and stability on forests and pseudoforests}

\author{%
Victor~N.~Schv\"ollner,
Adri\'an~Pastine
and
Daniel~A.~Jaume
}

\begin{document}

\maketitle




\begin{abstract}
Given any two forests (pseudoforests) with the same degree sequence, we show that one can be transformed into the other by a sequence of 2-switches in such a way that all the intermediate graphs of the transformation are forests (pseudoforests). We also prove that the 2-switch operation perturbs minimally some well-known integer parameters in families of graphs with the same degree sequence. Then, we apply these results to conclude that the studied parameters have the interval property in those families.
\end{abstract}


\section{Introduction}

Let $G$ be a graph. We use $V(G)$ and $E(G)$ to refer to the vertex set and the edge set of $G$, respectively. The degree sequence of a graph $G$ with $n$ vertices is denoted by $d=d(G)=(d_1,\ldots,d_n)$, where $d_v$ is the degree of the vertex $v$ in $G$. Let $a,b,c$ and $d$ be four distinct vertices of $G$ such that $ab,cd\in E(G)$ and $ac,bd\notin E(G)$. The process of deleting the edges $ab$ and $cd$ from $G$ and adding $ac$ and $bd$ to $G$ is referred to as a \emph{2-switch} on $G$. This is a classical operation, see \cite{Berge, Chartrand1996}. If $G'$ is the graph obtained from $G$ by a 2-switch, it is straightforward to check that $d(G)=d(G')$. In other words, this operation preserves the degree sequence. A key fact about 2-switch is the following.

\begin{theorem}\label{BergeTeo}
	If $G$ and $H$ are two graphs with the same degree sequence, then there exists a 2-switch sequence transforming $G$ into $H$.
\end{theorem}

\Cref{BergeTeo} appears throughout the literature; the switch operation and this result trace back to Petersen~\cite{Petersen}, and were subsequently rediscovered several times (e.g. by Senior~\cite{Senior}, who called the operation \emph{transfusion}). A similar result is known for bipartite graphs with a given bipartite degree sequence. For bipartite degree sequences, the fact that any realization can be transformed into any other by swaps is due to Ryser~\cite{Ryser} (see also Gale~\cite{Gale}); it can also be obtained from a bipartite version of the Havel-Hakimi theorem (\cite{Hakimi}, \cite{Havel}; see~\cite{West}).

The \emph{realization graph} of a graphical sequence $d$ (i.e., a sequence of integers $d$ which is the degree sequence of some graph) is the graph ${\cal G}(d)$ whose vertices are the
graphs with degree sequence $d$, with two graphs $G,H$ adjacent if $G$ and $H$ are obtained from
one another by a $2$-switch (see \cite{arikati1999realization}). A direct consequence of \Cref{BergeTeo} is that $\mathcal{G}(d)$ is connected. In recent years, the problem of determining whether a given induced subgraph of $\mathcal{G}(d)$ is connected or not has been extensively studied. In this direction, many advances have already been made. Since \Cref{BergeTeo} assures the existence of a path between $G$ and $H$ in $\mathcal{G}(d)$, a natural question is how short that path can be. The length of the shortest (unconstrained) 2-switch sequence was determined by Erd\H{o}s, Kir\'aly and Mikl\'os~\cite{ErdosKiralyMiklos} via a Gallai-type identity; an equivalent characterization was later obtained in~\cite{BeregIto}. The analogous question under the constraint that every intermediate graph be connected was answered by Taylor~\cite{Taylor} (existence) and by Fernandes~\cite{Fernandes} (shortest length).

\begin{theorem}[\cite{Taylor}]\label{Teo:Tay}
	Let $\mathcal{K}(d)$ be the subgraph of $\mathcal{G}(d)$ induced by connected graphs. Then, $\mathcal{K}(d)$ is connected.

\end{theorem}

The article by Taylor includes similar results for multigraphs, and asks whether a result similar to \Cref{Teo:Tay} can be obtained if we change the restricting property.

By $\mathcal{F}(d)$, $\mathcal{U}(d)$ and $\mathcal{P}(d)$, we denote the subgraphs of $\mathcal{G}(d)$ induced, respectively, by forests, unicyclic graphs and pseudoforests (i.e., graphs whose components are trees or unicyclic graphs). Notice that if $V(\mathcal{G}(d))$ contains a tree, then every connected graph in $V(\mathcal{G}(d))$ must be a tree. In such a case, $\mathcal{F}(d)$ is connected by \Cref{Teo:Tay}. Similarly, if $V(\mathcal{G}(d))$ contains a unicyclic graph, then every connected graph in $V(\mathcal{G}(d))$ must be unicyclic, and hence $\mathcal{U}(d)$ is connected by \Cref{Teo:Tay}. The most important feature of a connected induced subgraph $\mathcal{H}$ of $\mathcal{G}(d)$ is that it ensures the possibility of transforming, via 2-switches, a graph $G\in V(\mathcal{H})$ into another graph $G'\in V(\mathcal{H})$ in such a way that every intermediate graph in the transition also belongs to $V(\mathcal{H})$. This article is essentially about the connectedness of $\mathcal{F}(d)$ and $\mathcal{P}(d)$ (first part), and how this topological property is related to the values attained by integer parameters on $V(\mathcal{G}(d))$ (second part).

The first part of our work is organized as follows. In \Cref{sectiontswitch} we characterize those 2-switches that preserve tree or forest structure. Then, in \Cref{sectionFTT}, we show that $\mathcal{F}(d)$ is connected for all $d$, giving an algorithm to compute the transforming 2-switch sequence. Furthermore, in the same section, we establish an upper bound for the distance in $\mathcal{F}(d)$ between two forests, in terms of their sizes. In \Cref{sec-u-switch,sec-p-switch} we characterize those 2-switches that preserve, respectively, unicyclic and pseudoforest structure. Then, in \Cref{sec-pesudoforest}, we prove that $\mathcal{P}(d)$ is connected. Finally, in \Cref{sec_bipartite_nonbip}, we show that the subgraph of $\mathcal{G}(d)$ induced by bipartite (non-bipartite) graphs is not connected in general.

One of the most studied problems in the literature is, given a graph parameter (clique number, domination number, matching number, etc.), finding the minimum and maximum values for the parameter on $V(\mathcal{G}(d))$ (see \cite{BockRat,GHR1,GHR2,wang2008extremal,zhang2008laplacian,zhang2013number}). Another interesting problem is deciding which values between the minimum and the maximum can be realized by a graph in $V(\mathcal{G}(d))$ (see \cite{kurnosov2020set,Rao}). For instance, in \cite{BockRat} the authors study the matching number of trees with a given degree sequence, and of bipartite graphs with a given bipartite degree sequence. The authors find minimum and maximum values for the matching number in these families, and then show that all the other intermediate values are also attained. To establish the same property for bipartite degree sequences, they show that a 2-switch alters the matching number by at most $1$, and use the version of \Cref{BergeTeo} for bipartite degree sequences. In the case of trees, the authors of \cite{BockRat} constructed the tree realizing each value between the minimum and the maximum, although they could have applied \Cref{Teo:Tay}.

In the second part of this article, consisting only of  \Cref{sectionIntervalproperty}, we apply our results from the first part to show that a plethora of integer parameters are stable under 2-switch, i.e., the 2-switch perturbs them by at most 1, and have the interval property, which means that they attain every intermediate value between maximum and minimum. \Cref{sectionIntervalproperty} is organized as follows. In \Cref{sectionmatchingnumber,sectionindependencenumber} we study, respectively, matching number and independence number, obtaining results also on related parameters like edge-covering number, rank, nullity, vertex-covering number and clique number. Next, we study the domination number in \Cref{sectiondominationnumber} and the number of connected components in \Cref{sectionnumberofcomponents}. \Cref{sectionpathcover} is about path-covering number, zero forcing number and Z-Grundy domination number. Finally, in \Cref{sectionchromatic}, we analyze the chromatic number. Throughout this paper the sequence $d$ will be a graphical sequence.

\section{t-switch and f-switch}\label{sectiontswitch}

We denote a 2-switch operation by the $2\times 2$ matrix ${{a \ b}\choose{c \ d}}$, where the rows represent the edges to be removed and the columns represent the edges to be added. If $\tau= {{a \ b}\choose{c \ d}}$ is a 2-switch on $G$, then $\tau(G)$ denotes the transformed graph
\[ (G-\{ab,cd\})+\{ac,bd\}. \]

In order to prove that $\mathcal{F}(d)$ is connected, we need to characterize those 2-switches on a forest that preserve the forest structure. A first step in doing so is characterizing those preserving the tree structure. A 2-switch $\tau$ on a tree $T$ is said to be a \emph{t-switch} if $\tau(T)$ is a tree. It is easy to check that a 2-switch $\tau={{a \ b}\choose{c \ d}}$ on $T$ is a t-switch if and only if $T$ contains the path $ab\ldots cd$ or $ba\ldots dc$. Note that any two forests with the same degree sequence have the same number of connected components. In particular, if $d$ is the degree sequence of a tree, then every member of $V(\mathcal{F}(d))$ is also a tree.

A 2-switch $\tau$ on a forest $F$ is said to be an \emph{f-switch} if $\tau(F)$ is a forest. Using t-switch, we can easily characterize when a 2-switch on a forest is an f-switch. In fact, if $\tau$ is a 2-switch between two (disjoint) edges $e_{1}$ and $e_{2}$ of a forest $F$, we have the following:
\begin{enumerate}
	\item if $e_{1}$ and $e_{2}$ are in the same component $T$ of $F$, then $\tau$ is an f-switch on $F$ if and only if it is a t-switch on $T$;

	\item if $e_{1}$ and $e_{2}$ are in different components of $F$, then $\tau$ is an f-switch on $F$.
\end{enumerate}

\section{Forests}\label{sectionFTT}

Let \(G, H \in V(\mathcal{G}(d))\). Obviously, $G$ and $H$ have the same set of leaves. A leaf $\ell$ is said to be \emph{trimmable} in $G$ and $H$ if $\ell v\in E(G)\cap E(H)$, for some $v$. We denote the set of trimmable leaves of \(G\) and \(H\) by \(\Lambda(G,H)\), or just \(\Lambda\) when \(G\) and \(H\) are clear from the context. Recall that in this section we show that $\mathcal{F}(d)$ is connected. The next two lemmas are preliminary steps to this result, which will use the simple idea of consecutive deletions of trimmable leaves.

\begin{lemma}
	\label{tauiL}
	Let $F,F'\in V(\mathcal{F}(d))$. Suppose that $\theta=(\tau_{i})_{i=1}^{r}$ is an f-switch sequence transforming $F-\Lambda$ into $F'-\Lambda$. Then, $\theta$ is an f-switch sequence transforming $F$ into $F'$.
\end{lemma}
\begin{proof}
Let $F_0=F$ and $F_i=\tau_i(F_{i-1})$. As none of the vertices in $\Lambda$ is involved in any of the f-switches, $F_i-\Lambda=\tau_i(F_{i-1}-\Lambda)$ for all $i$. Hence $\theta$ is a sequence of 2-switches transforming $F$ into $F'$. Since $F_i-\Lambda$ is obtained from $F_i$ by removing vertices of degree $1$, $F_i-\Lambda$ has the same
cycles as $F_i$. Thus, as every $F_i-\Lambda$ is a forest, every $F_i$ is a forest as well. Therefore, $\theta$ is a sequence of f-switches transforming $F$ into $F'$.
\end{proof}

\begin{lemma}\label{ponerhojascompartidas}
	Let $F,F'\in V(\mathcal{F}(d))$. If \(\Lambda(F,F')=\varnothing\), then there exists an f-switch \(\tau\) on \(F\) such that  \(\Lambda(\tau(F),F')\neq \varnothing\).
\end{lemma}

\begin{proof}
We split the proof in two cases: 1) there is a leaf whose neighbor in $F'$  has degree $\geq 2$; 2) every vertex has degree 1.

\begin{enumerate}[label=(\arabic*).]
	\item Let $\ell$ be a leaf such that its neighbor $u$ in $F'$ has degree at least $2$,  and let $v$ be the neighbor of $\ell$ in $F$. If $\ell$ and $u$ are in different components of $F$, let $w$ be a neighbor of $u$ in $F$. If $\ell$ and $u$ are in the same component of $F$, consider the path $\ell v \ldots u$ in $F$. As $d_u \geq 2$, there is a neighbor $w$ of $u$ that is not in  $\ell v \ldots u$.

	In either case, $\tau={{\ell \ v}\choose{u \ w}}$ is an f-switch on $F$ such that $\ell\in \Lambda(\tau(F),F')$.

	\item Let $\ell$ be any leaf, let $v$ and $u$ be the neighbors of $\ell$ in $F$ and $F'$ respectively,
	and let $w$ be the neighbor of $u$ in $F$.
	Then, $\tau={{\ell \ v}\choose{u \ w}}$ is an f-switch on $F$ such that $\ell\in \Lambda(\tau(F),F')$. \qedhere
\end{enumerate}

\end{proof}

The next result states that given two forests with the same degree sequence, there is a sequence of f-switches transforming one into the other.
Before proceeding with the proof, we need to note two things. First, it is sufficient to prove it for forests without isolated vertices, because they do not participate in any 2-switch. Second, it is easy to check that the result holds for forests of order $n\leq 4$. We are now ready to proceed.

\begin{theorem}
	\label{IPTT}
	$\mathcal{F}(d)$ is connected.

\end{theorem}
\begin{proof}
Let $F,F'\in V(\mathcal{F}(d))$ and suppose that $F$ and $F'$ have no isolated vertices. We use induction on $n=|V(F)|$. If $n\leq 4$, the statement is true. Hence, let $n>4$, and suppose that every pair of forests of order $<n$ with the same degree sequence can be transformed into each other by a sequence of f-switches. We have two cases: 1) $\Lambda(F,F')\neq \varnothing$; 2) $\Lambda(F,F')= \varnothing$.
\begin{enumerate}[label=(\arabic*).]
	\item Consider $F-\Lambda$ and $F'-\Lambda$. These are two forests of order $n'<n$, with the same  degree sequence $d'$. So, the inductive hypothesis applies to $F-\Lambda$ and $F'-\Lambda$: there exists an f-switch sequence $\theta$ transforming $F-\Lambda$ into $F'-\Lambda$. Hence, by \Cref{tauiL}, $\theta$ is an f-switch sequence from $F$ to $F'$.

	\item By \Cref{ponerhojascompartidas}, this case can be reduced to case (1). \qedhere
\end{enumerate}
\end{proof}

The proof of \Cref{IPTT} contains a procedure that returns a transforming f-switch sequence. We make this algorithm explicit, see \eqref{algtrans} (Transition Algorithm).
\begin{algorithm}[h]
	\label{algtrans}
	{Transition Algorithm}\\
	INPUT: two forests $F,F' \in \mathcal{F}(d)$.
	 \begin{enumerate}

		\item Let $r=0$ and $\Lambda=\Lambda(F,F')$.
		\item While \(F\neq F'\):
			\begin{enumerate}
			\item If $\Lambda=\varnothing$:
				\begin{enumerate}
					\item Let \(r=r+1\).
					\item If every vertex in $F'$ has degree $1$ choose a leaf \(\ell \in V(F')\). Else choose a leaf
					$\ell \in V(F')$ such that its neighbor \(u\) in \(F'\) has degree at least $2$.
					\item Find the f-switch \(\tau\) such that \(\ell\) is trimmable between \(\tau(F)\) and \(F'\) (such f-switch exists by \Cref{ponerhojascompartidas}).
					\item Let \(\tau_{r}=\tau\), \(F=\tau(F)\), and $\Lambda=\Lambda(F,F')$.
				\end{enumerate}
			\item Let \(F=F-\Lambda\) and \(F'=F'-\Lambda\).
			\end{enumerate}
		\item RETURN \((\tau_{i})_{i=1}^{r}\).
	\end{enumerate}
\end{algorithm}
By the Transition Algorithm \eqref{algtrans}, either \(F=F'\) or at least one leaf is removed (equivalently, one edge is removed). Thus, the Transition Algorithm runs at most \(|E(F')-E(F)|\) times. This number can be improved by $1$. First, we need a technical lemma.

\begin{lemma}
	\label{etaneq1}
	If \(G,H \in V(\mathcal{G}(d))\), then \(| E(G) - E(H)|\neq 1\).
\end{lemma}
\begin{proof}
We proceed by contradiction. Assume that \(| E(G) - E(H)|= 1\). Clearly, $d(G)=d(H)$ implies $|E(G)|=|E(H)|$. Thus, $|E(H) - E(G)|= 1$. Let $ab\in E(G)-E(H)$. As $d(G)=d(H)$, there must be an edge incident to $a$ and an edge incident to $b$ in $E(H)-E(G)$. But, as $|E(H) - E(G)|= 1$, this can only be possible if $ab\in E(H)$, contradicting the fact that $ab$ is the only edge in $E(G)-E(H)$.
\end{proof}

Suppose now that $(\tau_{i})_{i=1}^{r}$ is an f-switch sequence obtained as output of the Transition Algorithm \eqref{algtrans} applied to \(F\) and \(F'\). Let $F_0=F$ and $F_i=\tau_i(F_{i-1})$ for $1\leq i \leq r$.
Notice that for each $i$,  $|E(F')\cap E(F_i)|\geq |E(F')\cap E(F_{i-1})|+1$. Thus,
\[
\sum_{i=1}^{r-1}\left( |E(F')\cap E(F_i)|-|E(F')\cap E(F_{i-1})| \right) \geq r-1.
\]
Since this last sum is telescoping we have
\[
|E(F')\cap E(F_{r-1})|-|E(F')\cap E(F)|\geq r-1.
\]
On the other hand, \Cref{etaneq1} implies
\begin{align*}
|E(F')\cap E(F_{r})|-|E(F')\cap E(F_{r-1})|=|E(F')|-|E(F')\cap E(F_{r-1})|\geq 2.
\end{align*}
Thus, 
\begin{align*}
|E(F')\cap E(F_{r})|-|E(F')\cap E(F)|\geq r+1.
\end{align*}
But $E(F_{r})=E(F')$, and $|E(F')|-|E(F')\cap E(F)|=|E(F')-E(F)|$. Therefore, $r\leq |E(F')-E(F)|-1$.
Hence, the Transition Algorithm runs at most $|E(F')-E(F)|-1$ times. We have just established the other main result of this section.
\begin{theorem}
	\label{transbosques}
	The distance between $F$ and $F'$ in $\mathcal{F}(d)$ is at most
	\[ |E(F')-E(F)|-1. \]
\end{theorem}

\begin{proof}
	It follows from \Cref{IPTT} and the previous discussion.
\end{proof}

It is worth noting that \Cref{transbosques} bounds the distance in $\mathcal{F}(d)$, where every intermediate graph must be a forest. The (unconstrained) 2-switch distance in $\mathcal{G}(d)$ was determined by Erd\H{o}s, Kir\'aly and Mikl\'os~\cite{ErdosKiralyMiklos} to be $|E(F')-E(F)|-\max C$ (where $\max C$ denotes the maximum number of $F$/$F'$-alternating circuits in $E(F)\triangle E(F')$), with $\max C\ge 1$, and can therefore be smaller than our bound. However, their shorter sequence need not consist of f-switches. Moreover, computing the exact 2-switch distance is NP-hard~\cite{BeregIto}, whereas \Cref{transbosques} provides an explicit, easily computable bound.

\section{u-switch} \label{sec-u-switch}

Let $G$ be a graph. By \(\Cycles(G)\) we denote the subgraph of $G$ induced by all the vertices that belong to some cycle of \(G\). Now, define
\[ \Forest(G) = G - E(\Cycles(G)), \]
which is a spanning forest of $G$ whose components are trees attached to vertices of \(\Cycles(G)\). Clearly, $E(G)=E(\Cycles(G))\dot{\cup} E(\Forest(G))$. A 2-switch $\tau$ on a unicyclic graph $U$ is said to be a \emph{u-switch} if $\tau(U)$ is also unicyclic.

\begin{lemma}
    \label{lem.tswitch.U-e}
    Let $U$ be a unicyclic graph and let $e\in E(\Cycles(U))$. If $\tau$ is a t-switch on $U-e$ between the edges $ab,cd$ of $\Forest(U)$, then $e\notin\tau(U-e)$.
\end{lemma}
\begin{proof}
	Let $C=\Cycles(U)$ be the cycle of $U$.	Notice that $C-e$ is a subgraph of both $U-e$ and $\tau(U-e)$, because none of the edges in $C-e$ are involved in $\tau$. If $e\in \tau(U-e)$, then $C\subseteq\tau(U-e)$ (the symbol $\subseteq$ denotes the subgraph relation). Hence, $\tau$ is not a t-switch on $U-e$.
\end{proof}

\begin{lemma}
    \label{lem2.tswitch.U-e}
    Let $U$ be a unicyclic graph and let $ab\in E(\Cycles(U))$, $cd\in E(\Forest(U))$. If $\tau$ is a 2-switch between $ab$ and $cd$, then there exists an edge $e\in E(\Cycles(U))$ such that $e\neq ab$ and $e\notin\tau(U-e)$.
\end{lemma}

\begin{proof}
	Let $\tau={{a \ b}\choose{c \ d}}$.
	Since $cd\in E(\Forest(U))$, at least one of the vertices $c, d$ does not belong to $\Cycles(U)$. Assume without loss of generality that $d\not\in \Cycles(U)$. Let $v\neq a$ be a neighbor of $b$ in $\Cycles(U)$, and let $e=bv$. Notice that $v\neq d$ because $d\not\in \Cycles(U)$. Therefore,
	\[e\not\in \tau(U-e)=((U-e)-\{ab,cd\})+\{ac,bd\}.\]\qedhere
\end{proof}

The following observation will be used many times. Let $\tau= {{a \ b}\choose{c \ d}}$ be a 2-switch on a graph $G$ and let $e\in E(G)-\{ab,\,cd\}$. If $e\notin\tau(G-e)$, then $\tau(G-e)+e=\tau(G)$. In particular, if $\tau(G-e)$ is a tree, then $\tau(G)$ is unicyclic. Next, we characterize the u-switches.

\begin{theorem}
	\label{uswitchcaract}
	Let $\tau$ be a 2-switch between two disjoint edges $e_{1}, e_{2}$ of a unicyclic graph $U$. Then, the following statements hold:
	\begin{enumerate}
		\item If $e_{1},e_{2}\in E(\Forest(U))$, then \(\tau\) is a u-switch on $U$ if and only if $\tau$ is a t-switch on $U-e$, for all $e\in E(\Cycles(U))$.

		\item If $e_{1}\in E(\Cycles(U))$ and $e_{2}\in E(\Forest(U))$, then \(\tau\) is a u-switch on $U$.

		\item If $e_{1},e_{2}\in E(\Cycles(U))$, then \(\tau\) is a u-switch on $U$ if and only if $\tau(\Cycles(U))\approx \Cycles(U)$.
	\end{enumerate}
\end{theorem}

\begin{proof}
	\begin{enumerate}[label=(\arabic*).]
		\item (\(\Leftarrow\)) By \Cref{lem.tswitch.U-e}, \(e \notin \tau(U-e)\). Since \(\tau(U-e)\) is a tree, \(\tau(U-e)+e=\tau(U)\) is a unicyclic graph.

		(\(\Rightarrow\))  If $e_{1},e_{2}\in E(\Forest(U))$, $e\in E(\Cycles(U))$ and $\tau$ is not a t-switch on $U-e$, then \(\tau\) disconnects $U$.

		\item By \Cref{lem2.tswitch.U-e}, there exists \(e \in E(\Cycles(U))-e_{1}\) such that \(e\notin\tau(U-e)\). Clearly, $U-e$ is a tree. If $\tau(U-e)$ is a tree, then \(\tau(U-e)+e=\tau(U)\) is a unicyclic graph. Otherwise, $\tau(U-e)$ has two components: a unicyclic graph $U'$ and a tree $T$. Since $e$ links $U'$ to $T$, we have that $\tau(U-e)+e=\tau(U)$ is a unicyclic graph.

		\item Let $C$ be a cycle. If $|V(C)|=3$, we cannot apply any 2-switch on $C$. If $|V(C)|\in\{4,5\}$, then $\tau(C)\approx C$ (the symbol $\approx$ denotes the isomorphism relation), for every 2-switch $\tau$ on $C$. If $|V(C)|\geq 6$, then either $\tau(C)\approx C$ or $\tau(C)$ is the union of two disjoint cycles.

		(\(\Rightarrow\)) If $\tau(\Cycles(U))\not\approx \Cycles(U)$, then \(\tau\) disconnects $U$.

		(\(\Leftarrow\)) If $\tau(\Cycles(U))\approx \Cycles(U)$, then $\tau(U)$ is obviously unicyclic. \qedhere
	\end{enumerate}
\end{proof}

Recall that, if $V(\mathcal{F}(d))$ contains a tree, then all members of $V(\mathcal{F}(d))$ are trees as well. In contrast, if $V(\mathcal{P}(d))$ contains a unicyclic graph, then the rest of the pseudoforests in $V(\mathcal{P}(d))$ are not necessarily all unicyclic.

\section{p-switch}\label{sec-p-switch}

A 2-switch $\tau$ on a pseudoforest $G$ is said to be a \emph{p-switch} if $\tau(G)$ is a pseudoforest. Notice that t-switches, f-switches and u-switches are clearly particular cases of p-switches.

\begin{lemma}
	\label{lem2switchpswitch}
	Let $G$ be a pseudoforest with two components, $F$ and $U$, where $F$ is a forest and $U$ is a unicyclic graph. We have the following:
	\begin{enumerate}
		\item every 2-switch on $F$ is a p-switch;

		\item let $\tau= {{a \ b}\choose{c \ d}}$ be a 2-switch on $G$. If $ab\in E(F)$ and $cd\in E(\Forest(U))$, then $\tau$ is a p-switch on $G$.
	\end{enumerate}
\end{lemma}

\begin{proof}
	\begin{enumerate}[label=(\arabic*).]
		\item It is straightforward to see that every 2-switch on $F$ creates at most one cycle.

		\item Choose any $e\in E(\Cycles(U))$ and notice that $G-e$ is a forest. Since $ab$ and $cd$ are in different components, $\tau(G-e)$ is a forest and so $\tau(G-e)+e$ contains at most one cycle. Since $|e\cap \{a,b,c,d\}|\leq 1$, $e\notin\tau(G-e)$. Hence, $\tau(G-e)+e=\tau(G)$. \qedhere
	\end{enumerate}
\end{proof}

\begin{lemma}
	\label{2switch.in.C_U=pswitch.in.U}
     If $U$ is a unicyclic graph, then every 2-switch between two edges of $\Cycles(U)$ is a p-switch on $U$.
\end{lemma}

\begin{proof}
	Let $\tau$ be a 2-switch on $U$. If $\tau$ is a u-switch, then $\tau$ is obviously a p-switch. Otherwise, notice that $\tau(U)$ consists of two unicyclic components, and so it is a pseudoforest.
\end{proof}

Let $G$ be a graph. We denote by $\cycles(G)$ the number of subgraphs of $G$ isomorphic to a cycle. We say that $G$ has \emph{cyclicity} $c(G)$ if $\cycles(H)\leq c(G)$ for every component $H$ of $G$. Pseudoforests are exactly the graphs with cyclicity $\leq 1$.

\begin{lemma}
	\label{lema2switchpswitch}
	Let $\tau= {{a \ b}\choose{c \ d}}$ be a 2-switch on a pseudoforest $G$ with $c(G)=1$. Suppose that one of the following conditions holds:
	\begin{enumerate}
		\item $ab\in E(\Forest(G))$ and $cd\in E(\Cycles(G))$;

		\item $ab,cd\in E(\Cycles(G))$.
	\end{enumerate}
	Then, $\tau$ is a p-switch on $G$.
\end{lemma}

\begin{proof}
	For each case of the hypothesis we have the following subcases:
	(A) $ab$ and $cd$ lie in the same component of $G$;
	(B) $ab$ and $cd$ lie in different components of $G$.
	\begin{enumerate}
		\item[(1.A)] Use \Cref{uswitchcaract}.

		\item[(2.A)] Use \Cref{2switch.in.C_U=pswitch.in.U}.

		\item[(1.B)] Let $H$ be the component of $G$ containing $ab$, and let $U$ be the component of $G$ containing $cd$. Since $cd \in E(\mathrm{Cyc}(G))$, the component $U$ is unicyclic. Since $ab \in E(\mathrm{For}(G))$, the edge $ab$ does not lie on any cycle of $G$; in particular, $ab$ is a bridge of $H$. By the case hypothesis, $H \ne U$.

		We analyze $\tau(H \,\dot\cup\, U)$ directly. Since $ab$ is a bridge, $H - ab$ has exactly two components: if $H$ is a tree, both are trees; if $H$ is unicyclic with cycle $\mathrm{Cyc}(H)$, one of them contains $\mathrm{Cyc}(H)$ and is therefore unicyclic, while the other is a tree. Let $H_a$ and $H_b$ denote the components of $H - ab$ containing $a$ and $b$, respectively. On the other hand, since $cd \in E(\mathrm{Cyc}(U))$, the graph $T_U = U - cd$ is a tree.

		The 2-switch $\tau$ adds $ac$ and $bd$ to $H_a \,\dot\cup\, H_b \,\dot\cup\, T_U$. Since $a \in V(H_a)$, $b \in V(H_b)$, $c, d \in V(T_U)$, and $V(H_a)$, $V(H_b)$, $V(T_U)$ are pairwise disjoint, the edges $ac$ and $bd$ link the three pieces in a path-like fashion ($H_a$---$T_U$---$H_b$) without creating cycles among them. Hence $\tau(H \,\dot\cup\, U)$ is connected and contains at most one cycle (namely $\mathrm{Cyc}(H)$, if $H$ is unicyclic). The remaining components of $G$ are not affected by $\tau$. Therefore $\tau(G)$ is a pseudoforest.

		\item[(2.B)] Note that $\tau$ glues the two cycles containing $ab$ and $cd$ together into a new cycle. Hence, $c(\tau(G)) = c(G)$. \qedhere
	\end{enumerate}
\end{proof}

The next theorem characterizes when a 2-switch transforms a pseudoforest into another pseudoforest.

\begin{theorem}
	\label{pswitchcaract}

	Let $\tau= {{a \ b}\choose{c \ d}}$ be a 2-switch on a pseudoforest $G$. Then, the following statements hold.
	\begin{enumerate}
		\item If \(ab\) and \(cd\) are in different components of \(\Forest(U)\), for some unicyclic component \(U\) of \(G\), then \(\tau\) is a p-switch if and only if it is a t-switch on \( U-e\) for all \(e \in E(\Cycles(U))\).

		\item If \( ab \in E(\Forest(U))\) and \(cd \in E(\Forest(U'))\), for some distinct unicyclic components \(U\) and \( U'\) of \(G\), then \(\tau\) is a p-switch if and only if

		the components of $(U\dot{\cup}U')-\{ab,cd\}$ containing $b$ and $c$ are both trees.

		\item In any other case, \(\tau\) is a p-switch.
	\end{enumerate}
\end{theorem}

\begin{proof}
	\begin{enumerate}[label=(\arabic*).]
		\item By \Cref{uswitchcaract}, $\tau$ is a u-switch on $U$. Hence, it is a p-switch on $G$.

		\item Straightforward.

		\item The remaining cases are covered by \Cref{lem2switchpswitch} and \Cref{lema2switchpswitch}. \qedhere
	\end{enumerate}
\end{proof}

\section{Pseudoforests} \label{sec-pesudoforest}

By $\kappa(G)$ we denote the number of components of a graph $G$.

\begin{lemma}
	\label{pseudosize}

	If $G$ is a pseudoforest, then $|E(G)|+\kappa(G)=|V(G)|+\cycles(G)$.
\end{lemma}

\begin{proof}
If we remove an edge from every cycle of $G$, then we obtain a forest $F$ such that $\kappa(F)=\kappa(G)$. Therefore, $|E(F)|=|V(G)|-\kappa(G)$. On the other hand, $|E(F)|=|E(G)|-\cycles(G)$ and hence $|E(G)|-\cycles(G)=|V(G)|-\kappa(G)$.
\end{proof}

\begin{proposition}
	\label{zeta_constant}
	The function $\zeta :V(\mathcal{P}(d))\rightarrow \mathbb{Z}$, defined by
	\[ \zeta(G)= \kappa(G)-\cycles(G), \]
	is a non-negative constant.
\end{proposition}

\begin{proof}
	By \Cref{pseudosize}, we have $|V(G)|-|E(G)|=\kappa(G)-\cycles(G)=\zeta(G)$. Since pseudoforests have at most one cycle per component, $\zeta\geq 0$. Since all vertices of $\mathcal{P}(d)$ have the same order and size, $\zeta$ is constant.
\end{proof}

\begin{corollary}\label{coro_ciclos}
	If $G,H\in V(\mathcal{P}(d))$, then $\cycles(G)=\kappa(G)$ if and only if $\cycles(H)=\kappa(H)$.
\end{corollary}

\begin{proof}
	Since $\zeta(G)=0$, $\zeta$ is the zero function by \Cref{zeta_constant}.
\end{proof}

\begin{lemma}
	\label{pseudo>>>unic}

	Every pseudoforest $G$ with $\cycles(G)=\kappa(G)$ can be transformed into a unicyclic graph by a sequence of p-switches.
\end{lemma}

\begin{proof}
	If $G$ is connected, we are done. If $\kappa(G)\geq 2$, then observe that we can link two components $U$ and $U'$ of $G$ by performing a 2-switch $\tau$ between $e_{1}\in E(\Cycles(U))$ and $e_{2}\in E(\Cycles(U'))$. By \Cref{pswitchcaract}, we know that $\tau$ is a p-switch on $G$. By the proof of \Cref{lema2switchpswitch}, we know that $\tau(U\dot{\cup}U')$ is a unicyclic graph. Now, $\kappa(\tau(G))=\kappa(G)-1$. Therefore, we repeat the process until we obtain a connected pseudoforest $H$. By \Cref{coro_ciclos}, $\cycles(H)=\kappa(H)=1$. Thus, $H$ is a unicyclic graph.
\end{proof}

\begin{lemma}
	\label{pseudo>>>forest}

	Every pseudoforest $G$ with $\cycles(G)<\kappa(G)$ can be transformed into a forest by a sequence of p-switches.
\end{lemma}

\begin{proof}
	Every pseudoforest $G$ with $\cycles(G)<\kappa(G)$ can be written as $G=H\dot{\cup}F$, where $F$ is a forest and $H$ is a pseudoforest such that each of its components is a unicyclic graph, i.e., $\cycles(H)=\kappa(H)$. Then, we can apply \Cref{pseudo>>>unic} to $H$ to obtain from $G$ a pseudoforest $G'=U\dot{\cup}F$, where $U$ is a unicyclic graph. Now, perform a 2-switch $\tau$ between $e_{1}\in E(F)$ and $e_{2}\in E(\Cycles(U))$. Then, $\tau$ is a p-switch by  \Cref{pswitchcaract} and $\tau(G')$ is a forest by the proof of \Cref{lema2switchpswitch}.
\end{proof}

\begin{theorem}
	\label{pseudoTransition}

	$\mathcal{P}(d)$ is connected.
\end{theorem}

\begin{proof}
	Let $G,H\in V(\mathcal{P}(d))$. If $\cycles(G)=\kappa(G)$, then $\cycles(H)=\kappa(H)$ as well, by \Cref{coro_ciclos}. Now, apply \Cref{pseudo>>>unic} to $G$ and $H$ to obtain respectively $U,U'\in V(\mathcal{U}(d))$. Since $\mathcal{U}(d)$ is connected, we can transform $U$ into $U'$ by a sequence of u-switches and hence we can transform $G$ into $H$ by a sequence of p-switches.

    If $\cycles(G)<\kappa(G)$, then  $\cycles(H)<\kappa(H)$ by \Cref{coro_ciclos} and \Cref{zeta_constant}. Now, apply \Cref{pseudo>>>forest} to $G$ and $H$ to obtain respectively $F,F'\in V(\mathcal{F}(d))$. By \Cref{IPTT}, we can transform $F$ into $F'$ by a sequence of f-switches and hence we can transform $G$ into $H$ by a sequence of p-switches.
\end{proof}

\section{Bipartite and non-bipartite graphs}\label{sec_bipartite_nonbip}

So far, we have only seen examples of connected induced subgraphs of $\mathcal{G}(d)$. A rather uninteresting example of a disconnected induced subgraph of $\mathcal{G}(d)$ could be obtained by finding two graphs $G,H \in V(\mathcal{G}(d))$ such that $|E(G)-E(H)| \geq 3$, and then considering the subgraph $\mathcal{X}$ induced by $\{G,H\}$. Since it is not possible to transform $G$ into $H$ via a single 2-switch (since otherwise, $|E(G)-E(H)|=2$), it follows that $\mathcal{X}$ has no edges. In this section, we consider the subgraph of $\mathcal{G}(d)$ induced by bipartite graphs and the subgraph of $\mathcal{G}(d)$ induced by non-bipartite graphs. In both cases, we show that these subgraphs are not connected in general.\\

For $n\geq 3$, consider the bipartite graph $B_n$ obtained by attaching two leaves to $K_{n,n}$, in such a way that the distance between them is 3 (i.e., one leaf in a part of the bipartition and one in the other; see \Cref{contrabipartitos}). Thus, $B_n$ has two vertices of degree $n+1$, two vertices of degree $1$ and the remaining $2n-2$ vertices of degree $n$. If $\tau$ is a 2-switch on $B_n$, we will show that $\tau(B_n )$ is non-bipartite or $\tau(B_n )\approx B_n$. Applying $\tau$ between the edges $e$ and $e'$, we can distinguish 3 cases:
\begin{enumerate}[label=(\arabic*).]
	\item $e$ is incident to a leaf but $e'$ is not. Let $(X,Y)$ be the bipartition of $V(B_n)$, and let $\ell v, xy\in E(B_n)$, with $v,x\in X$ and $d_{\ell}=1$. Then, apply $\tau={{\ell \ v}\choose{y \ x}}$ on $B_n$ and notice that $\tau(B_n)$ now contains a triangle $vxuv$ for every $u\in Y-y$.

	\item $e$ and $e'$ are not incident to leaves. Let $x,x',y,y'$ be vertices of $B_n$ with degree $\geq 2$. If $x,x'\in X$ and $y,y'\in Y$, perform $\tau={{x \ y}\choose{x' \ y'}}$ on $B_n$. Then, $\tau(B_n)$ contains triangles of the form $xx'ux$ and $yy'vy$, for every $u\in Y-\{y,y'\}$ and for every $v\in X-\{x,x'\}$.

	\item $e$ and $e'$ are both incident to leaves. Applying on $B_n$ the unique 2-switch between the edges incident to the leaves, we clearly obtain a graph isomorphic to $B_n$ (the systematic study of graphs $G$ for which every 2-switch yields a graph isomorphic to $G$ was carried out in Section 2 of \cite{BrualdiFernandesFurtado2019}).
\end{enumerate}
Thus, every 2-switch sequence transforming $B_n$ into a bipartite graph $B\not\approx B_n$ must pass through a non-bipartite graph. An example of such a graph $B$ is, for $n\geq 3$, the graph $B'_n$ obtained by attaching two leaves to $K_{n,n}$ in such a way that the distance between them is 4 (i.e., both leaves are in the same part of the bipartition; see \Cref{contrabipartitos}).

\begin{figure}[h]
	\[
	\begin{tikzpicture}
		[scale=.7,auto=left,every node/.style={scale=.6,circle,thick,draw}]
		\node [draw=none, scale=1.5] at (-1.5,0) {$B_3$};
		\node [label=above:,fill](1) at (1,2) {};
		\node [label=above:](2) at (3.5,2) {};
		\node [label=above:,fill](3) at (1,0) {};
		\node [label=below:](4) at (3.5,0) {};
		\node [label=below:,fill](5) at (1,-2) {};
		\node [label=below:](6) at (3.5,-2) {};
		\node [label=below:,fill](7) at (4.5,-1) {};
		\node [label=above:](8) at (0,1) {};
		\draw (1) -- (2) ;
		\draw (2) -- (3) ;
		\draw (3) -- (4);
		\draw (4) -- (5) ;
		\draw (5) -- (6) ;
		\draw (1) -- (8) ;
		\draw (7) -- (6) ;
		\draw (6) -- (1) ;
		\draw (1) -- (4) ;
		\draw (3) -- (6) ;
		\draw (2) -- (5);
		\node [draw=none, scale=1.5] at (7.5,0) {$B'_3$};
		\node [label=above:,fill](1b) at (9,2) {};
		\node [label=above:](2b) at (11.5,2) {};
		\node [label=above:,fill](3b) at (9,0) {};
		\node [label=below:](4b) at (11.5,0) {};
		\node [label=below:,fill](5b) at (9,-2) {};
		\node [label=below:](6b) at (11.5,-2) {};
		\node [label=below:,fill](7b) at (12.5,-1) {};
		\node [label=above:,fill](8b) at (12.5,1) {};
		\draw (1b) -- (2b) ;
		\draw (2b) -- (3b) ;
		\draw (3b) -- (4b) ;
		\draw (4b) -- (5b) ;
		\draw (5b) -- (6b) ;
		\draw (2b) -- (8b) ;
		\draw (7b) -- (6b) ;
		\draw (6b) -- (1b) ;
		\draw (1b) -- (4b) ;
		\draw (3b) -- (6b) ;
		\draw (2b) -- (5b);
	\end{tikzpicture}
	\]
	\caption{Every 2-switch sequence transforming $B_3$ into $B'_3$ must pass through
		a non-bipartite graph.}
	\label{contrabipartitos}
\end{figure}

Let $N_k$ be the non-bipartite graph with the following two components: a triangle $K_3$ together with a star $S_k$ of order $k$, for $k\geq 4$. Then, $N_k$ has $k+3$ vertices: one vertex of degree $k-1$, three vertices of degree 2 and the remaining $k-1$ vertices of degree 1. Observe that the only 2-switches we can perform on $N_k$ are between an edge of $K_3$ and an edge of $S_k$. Moreover, if $\tau$ is such a 2-switch, it is easy to see that $\tau(N_k)$ is always a tree, that is, a bipartite graph. Thus, every 2-switch sequence transforming $N_k$ into a non-bipartite graph $N\not\approx N_k$ must pass through a bipartite graph. As an example of such an $N$, we can take the graph $N'_k$ ($k\geq 4$), formed by 2 components: a path of order 3 and a triangle with $k-3$ leaves attached to one of its vertices (see \Cref{contranobipartitos}).

\begin{figure}[h]
	\[
	\begin{tikzpicture}
		[scale=.7,auto=left,every node/.style={scale=.6,circle,thick,draw}]
		\node [draw=none, scale=1.5] at (-0.75,0) {$N_4$};
		\node [label=above:](1) at (1,2) {};
		\node [label=above:](2) at (3.5,2) {};
		\node [label=above:](3) at (1,0) {};
		\node [label=below:](4) at (3.5,0) {};
		\node [label=below:](5) at (1,-2) {};
		\node [label=below:](6) at (3.5,-2) {};
		\node [label=below:](7) at (4.5,-1) {};
		\draw (1) -- (2) ;
		\draw (2) -- (3) ;
		\draw (3) -- (1);
		\draw (5) -- (6) ;
		\draw (7) -- (6) ;
		\draw (4) -- (6) ;
		\node [draw=none, scale=1.5] at (7.5,0) {$N'_4$};
		\node [label=above:](1b) at (9,2) {};
		\node [label=above:](2b) at (11.5,2) {};
		\node [label=above:](3b) at (9,0) {};
		\node [label=below:](4b) at (11.5,0) {};
		\node [label=below:](5b) at (9,-2) {};
		\node [label=below:](6b) at (11.5,-2) {};
		\node [label=below:](7b) at (12.5,-1) {};
		\draw (1b) -- (2b) ;
		\draw (2b) -- (3b) ;
		\draw (3b) -- (1b) ;
		\draw (4b) -- (3b) ;
		\draw (5b) -- (6b) ;
		\draw (7b) -- (6b) ;
	\end{tikzpicture}
	\]
	\caption{Every 2-switch sequence transforming $N_4$ into $N'_4$ must pass through
		a bipartite graph.}
	\label{contranobipartitos}
\end{figure}

\section{Stability and interval property}\label{sectionIntervalproperty}

Let $\mathcal{X}$ be a connected induced subgraph of $\mathcal{G}(d)$ and let \(\xi:V(\mathcal{X})\rightarrow \mathbb{Z}\). We say that \(\xi\) is \emph{stable} (under $2$-switch) in $\mathcal{X}$ if, for each $G\in V(\mathcal{X})$, we have
\[
\left| \xi(\tau(G))-\xi(G)\right| \leq 1,
\]
for every 2-switch $\tau$ on $G$ such that $\tau(G)\in V(\mathcal{X})$. Clearly, if $\xi$ is stable in $\mathcal{G}(d)$, then it is stable in every connected induced subgraph of $\mathcal{G}(d)$. When $\mathcal{X}=\mathcal{G}(d)$, we simply say that $\xi$ is stable. Let $G,H\in V(\mathcal{X})$. Notice that if $\xi$ is stable in $\mathcal{X}$, then
\begin{equation*}
	dist_{\mathcal{X}}(G,H)\geq |\xi(G)-\xi(H)|,
\end{equation*}
where $dist_{\mathcal{X}}(*,*)$ is the usual path-metric in ${\mathcal{X}}$. This inequality gives an interesting way to estimate $dist_{\mathcal{X}}(G,H)$ (hard to determine in general) by choosing a suitable stable parameter easier to compute on $G$ and $H$.

We say that $\xi$ has the \emph{interval property} in $\mathcal{X}$ if $\xi(V(\mathcal{X}))=I\cap \mathbb{Z}$, for some interval $I\subseteq \mathbb{R}$. In other words, if $\xi_{\max}$ and $\xi_{\min}$ are the maximum and minimum values attained by $\xi$, then $\xi$ has the interval property in $\mathcal{X}$ if, for every integer $k$ in the interval $[\xi_{\min}, \xi_{\max}]$, there is a $G\in V(\mathcal{X})$ such that $\xi(G)=k$. One could think of this property as a discrete analog of the Intermediate Value Theorem from elementary calculus. The next theorem shows that stability implies interval property. This idea was used in \cite{BockRat} to prove the interval property for the matching number in the family of bipartite graphs with a given bipartite degree sequence, but here we formalize it in order to apply it to many parameters.

\begin{theorem} \label{printparamgeneral}
	Let $\mathcal{X}$ be a connected induced subgraph of $\mathcal{G}(d)$ and let \(\xi:V(\mathcal{X})\rightarrow \mathbb{Z}\). If \(\xi\) is stable in $\mathcal{X}$, then \(\xi\) has the interval property in $\mathcal{X}$.

\end{theorem}

\begin{proof}
Consider two graphs $G_{1},G_{2}\in V(\mathcal{X})$ such that $\xi(G_{1})$ and $\xi(G_{2})$ are, respectively, the minimum and the maximum value for $\xi$. As $\mathcal{X}$ is connected, there exists a  2-switch sequence $(\tau_{i})$ transforming $G_{1}$ into $G_{2}$, such that every intermediate graph of the transformation is also a vertex of $\mathcal{X}$. Since each $\tau_{i}$ perturbs $\xi$ by at most 1, every integer value in the interval $[\xi(G_{1}),\xi(G_{2})]$ must be attained in some graph of the transition. Thus, $\xi$ has the interval property in $\mathcal{X}$.
\end{proof}

Note that \Cref{printparamgeneral} provides a way to obtain a forest $F$ with $\xi(F)=k$, provided that $\xi$ is stable in $\mathcal{F}(d)$, where $d=d(F)$. First, find two forests $F_{1},F_{2}\in V(\mathcal{F}(d))$ such that $\xi_{1}=\xi(F_{1})\leq k\leq\xi_{2}=\xi(F_{2})$. Next, apply the Transition Algorithm \eqref{algtrans} to transform $F_{1}$ into $F_{2}$. Then, the required $F$ is one of the intermediate forests of the transition.

Another important observation about \Cref{printparamgeneral} is that the converse is not true, i.e., interval property does not imply stability in general. An interesting counterexample is obtained by computing the diameter $\delta$ of trees with degree sequence $d=(3,2,2,2,2,2,2,1,1,1)$. One can easily see by inspection that $\{6,7,8\}$ are the only values attained by $\delta$ on $V(\mathcal{F}(d))$. So, $\delta$ has the interval property in $\mathcal{F}(d)$. However, there is a tree $T\in V(\mathcal{F}(d))$ and a t-switch $\tau$ on $T$ such that $\delta(\tau(T)) - \delta(T)=2$, which shows that $\delta$ is not stable in $\mathcal{F}(d)$ (see \Cref{contradiam}).

\begin{figure}[h]
	\[
	\begin{tikzpicture}
		[scale=.7,auto=left,
		every node/.style={scale=.6,circle,thick,draw},
		every label/.append style={scale=1.5}]
        \node [draw=none, scale=1.5] at (4.75,2.4) {$T$};
		\node [label=above:](1) at (2,1.5) {};
		\node [label=above:](2) at (4,1.5) {};
		\node [label=right:$a$, fill](3) at (6,1.5) {};
		\node [label=below:](4) at (2,0) {};
		\node [label=below:](5) at (4,0) {};
		\node [label=below:](6) at (6,0) {};
		\node [label=below:$b$,fill](10) at (7.5,0) {};
		\node [label=above:$d$,fill](7) at (2,-1.5) {};
		\node [label=above:$c$,fill](8) at (4,-1.5) {};
		\node [label=above:](9) at (6,-1.5) {};
		\draw (1) -- (2); \draw (2) -- (3); \draw[thick] (3) -- (10);
		\draw (6) -- (10); \draw (9) -- (10); \draw (4) -- (5);
		\draw (5) -- (6); \draw[thick] (7) -- (8); \draw (8) -- (9);
        \node [draw=none, scale=1.5] at (12.25,2.4) {$T'$};
		\node [label=above:](1b) at (9.5,1.5) {};
		\node [label=above:](2b) at (11.5,1.5) {};
		\node [label=right:](3b) at (13.5,1.5) {};
		\node [label=above:$a$,fill](4b) at (9.5,0) {};
		\node [label=above:$c$,fill](5b) at (11.5,0) {};
		\node [label=below:](6b) at (13.5,0) {};
		\node [label=below:$b$,fill](10b) at (15,0) {};
		\node [label=above:](7b) at (9.5,-1.5) {};
		\node [label=above:](8b) at (11.5,-1.5) {};
		\node [label=above:$d$,fill](9b) at (13.5,-1.5) {};
		\draw (1b) -- (2b); \draw (2b) -- (3b); \draw (3b) -- (10b);
		\draw (6b) -- (10b); \draw[thick] (9b) -- (10b); \draw[thick] (4b) -- (5b);
		\draw (5b) -- (6b); \draw (7b) -- (8b); \draw (7b) -- (4b);
	\end{tikzpicture}
	\]
	\caption{$\delta$ is not stable in $\mathcal{F}(d)$: $T'={{a\ b}\choose{c\ d}}T$,
		but $\delta(T')-\delta(T)=2$.}
	\label{contradiam}
\end{figure}

If $\xi$ is stable, recall that $\xi$ is stable in each connected induced subgraph $\mathcal{X}$ of $\mathcal{G}(d)$. Therefore, $\xi$ has the interval property in each connected induced subgraph of $\mathcal{G}(d)$, by \Cref{printparamgeneral}. In such a case, we simply say that $\xi$ has the interval property.

\begin{corollary}
	\label{stable_implies_int_prop}
	If $\xi$ is stable, then $\xi$ has the interval property.
\end{corollary}

\begin{proof}
	It follows from the previous discussion.
\end{proof}

The next lemma provides an easy way to prove that an integer parameter is stable in $\mathcal{X}$. We will use it several times in the next sub-sections.

\begin{lemma}\label{lemastable}
	Let $\mathcal{X}$ be a connected induced subgraph of $\mathcal{G}(d)$ and let $\xi:V(\mathcal{X})\rightarrow\mathbb{Z}$. Assume that one of the following inequalities holds for each graph $G\in V(\mathcal{X})$ and for every $2$-switch $\tau$ on $G$ such that $\tau(G)\in V(\mathcal{X})$:
	\begin{multicols}{2}
		\begin{enumerate}
			\item $\xi(\tau(G))\leq \xi(G)+1$;
			\item $\xi(\tau(G))\geq \xi(G)-1$.
		\end{enumerate}
	\end{multicols}
	Then, $\xi$ is stable in $\mathcal{X}$.
\end{lemma}

\begin{proof}
	Assume that (1) holds and apply (1) to the graph $H=\tau(G)$ and the 2-switch $\tau^{-1}$ that undoes $\tau$ (i.e., $\tau^{-1}(H)=G$). In this way, we obtain inequality (2). Finally, combining (1) and (2) yields $\left| \xi\left(\tau(G)\right)-\xi(G)\right|\leq 1$, which shows that $\xi$ is stable in $\mathcal{X}$. The proof is the same if we start assuming the inequality (2).
\end{proof}

\subsection{Matching number and related parameters}\label{sectionmatchingnumber}

A matching in a graph $G$ is a set of pairwise disjoint edges of $G$. The maximum size of a matching in $G$ is called the \emph{matching number} of $G$, which is denoted by $\mu(G)$. A matching in \(G\) with maximum size is called a maximum matching. A proof of the stability of $\mu$ under 2-switch can be found between the lines of \cite{BockRat}.
We include our proof for completeness.

\begin{lemma}
	\label{matchlem1}
	Let $M$ be a maximum matching in a graph $G$ and let $\tau$ be a 2-switch between $e_{1},e_{2}\in E(G)$. If $e_{1}$ and $e_{2}$ are both in $M$ or both in $E(G)-M$, then $\mu(G)\leq \mu(\tau (G))$.
\end{lemma}

\begin{proof}
If $e_{1},e_{2}\in E(G)-M$, then $M$ is also a matching in $\tau (G)$.  Hence, $|M|=\mu(G)\leq \mu(\tau (G))$.

If $e_{1},e_{2}\in M$, the set $M'=M-\{e_{1},e_{2}\}$ is a matching of size $\mu(G)-2$ in $\tau (G)=(G-\{e_{1},e_{2}\}) + \{e'_{1},e'_{2}\}$, where \(e'_{1}\) and \(e'_{2}\) are the edges that \(\tau\)  adds to \(G\). Notice that none of the four vertices involved in $\tau$ belongs to any edge of $M'$. Hence, \(M'\cup \{e'_{1},e'_{2}\}\) is a matching of \(\tau (G)\).  Therefore, $\mu(\tau(G))\geq |M'\cup \{e'_{1},e'_{2}\}|=(\mu(G)-2)+2=\mu(G)$.
\end{proof}

\begin{lemma}
	\label{matchlem2}
	 Let $M$ be a maximum matching in a graph $G$, and $\tau$ be a 2-switch between $e_{1},e_{2}\in E(G)$. If $e_{1}\in M$ and $e_{2}\notin M$, then $\mu(\tau(G))\geq\mu(G)-1$.
\end{lemma}

\begin{proof}
The set $M-e_{1}$ is a matching in $\tau(G)$ of size $\mu(G)-1$. Thus, $\mu(\tau(G))\geq \mu(G)-1$.
\end{proof}

\begin{theorem}
	\label{matching_stable}
	The matching number is stable.
\end{theorem}

\begin{proof}
Let $G$ be a graph and $\tau$ be a 2-switch on $G$. By \Cref{matchlem1} and \Cref{matchlem2}, we have $\mu(\tau(G))\geq \mu(G)-1$. Thus, $\mu$ is stable by \Cref{lemastable}.
\end{proof}

Let $G$ be a graph. An edge-cover of $G$ is a set of edges $\mathcal{E}$ such that every vertex of $G$ is incident to at least one edge of $\mathcal{E}$. A minimum edge-cover of $G$ is an edge-cover of $G$ of minimum size. The \emph{edge-covering number} of $G$, denoted by $\epsilon(G)$, is the size of a minimum edge-cover of $G$.

\begin{corollary}
	\label{edgecover_stable}
	The edge-covering number is stable.
\end{corollary}

\begin{proof}
	It follows from this fact: $\epsilon(G)=|V(G)|-\mu(G)$ (see \cite{Gallai}).
\end{proof}

\begin{theorem}
	The matching number and the edge-covering number have the interval property.
\end{theorem}

\begin{proof}
	It follows from \Cref{matching_stable} and \Cref{edgecover_stable,stable_implies_int_prop}.
\end{proof}

The \emph{rank} and \emph{nullity} of a graph $G$, denoted by $\rank(G)$ and $\mnull(G)$ respectively, are the rank and nullity of its adjacency matrix. It is known that $\rank(F)=2\mu(F)$, for any forest $F$ (see \cite{Bevis,Gutman}). Combining this fact with the rank-nullity theorem from linear algebra (i.e., $\rank(G)+\mnull(G)=|V(G)|$) and the interval property of $\mu$ in $\mathcal{F}(d)$, we get the following result.

\begin{corollary}\label{rank_null}
	Let $F$ be a forest and let \(\tau\) be an f-switch on \(F\). Then, \[|\rank(\tau(F))-\rank(F)|=|\mnull(\tau(F))-\mnull(F)|\in \{0,2\}.\]
\end{corollary}

\begin{proof}
	It follows from the previous discussion.
\end{proof}

\subsection{Independence number and related parameters}\label{sectionindependencenumber}

An independent set of a graph $G$ is a set of vertices in $G$, no two of which are adjacent. A maximum independent set in $G$ is an independent set of \(G\) with the largest possible cardinality. This cardinality is called the \emph{independence number} of $G$, and it is denoted by $\alpha(G)$.

\begin{theorem}\label{indep_stable}
	The independence number is stable.
\end{theorem}
\begin{proof}
Let $I$ be a maximum independent set in a graph $G$, $\tau={{a \ b}\choose{c \ d}}$ be a 2-switch on \(G\) and let $V_{\tau}=\{a,b,c,d\}$. Notice that $|I\cap V_{\tau}|\leq 2$.

If \(\left| I \cap  V_{\tau} \right| \leq 1\), then \(\alpha(\tau(G)) \geq |I|=\alpha(G)\),  because $I$ is an independent set in $\tau(G)$. We can easily conclude the same when $I\cap V_{\tau}=\{a,d\}$ or $I\cap V_{\tau}=\{b,c\}$.

If $I\cap V_{\tau}=\{a,c\}$, notice that $I$ is not an independent set in $\tau(G)$, since  $ac\in E(\tau(G))$. Thus, $I-a$ is an independent set in $\tau(G)$, and so $\alpha(\tau(G))\geq |I-a|=\alpha(G)-1$. The same argument holds if $I\cap V_{\tau}=\{b,d\}$.

In any case, $\alpha(\tau(G))\geq \alpha(G)-1$. Thus, $\alpha$ is stable by \Cref{lemastable}.
\end{proof}

A vertex-cover of a graph $G$ is a set of vertices $U\subseteq V(G)$ such that each edge of $G$ is incident to at least one vertex of the set $U$. A minimum vertex-cover of $G$ is a vertex cover of $G$ of minimum size. The \emph{vertex-covering number} of $G$, denoted by $\nu(G)$, is the size of a minimum vertex-cover of $G$.

\begin{corollary}
	\label{vertex_cover_stable}
	The vertex-covering number is stable.
\end{corollary}

\begin{proof}
	It follows from this fact: $\nu(G)=n-\alpha(G)$ (see \cite{Gallai}).
\end{proof}

Let $G$ be a graph. A clique is a subset of $V(G)$ that induces a complete subgraph of $G$. The \emph{clique number} of $G$, denoted by $\omega(G)$, is the maximum size of a clique in $G$. Recall that $\omega(G)=\alpha(\overline{G})$ (see Chapter 5 of \cite{Diestel2005}).

\begin{theorem}
	\label{clique_stable}
	The clique number is stable.
\end{theorem}
\begin{proof}
	If $G$ is a graph and $d=d(G)$, let $\overline{d}$ be the degree sequence of the complement $\overline{G}$ of $G$. Assume that $\omega$ is not stable in $\mathcal{G}(d)$. Then, there exist a graph $H\in V(\mathcal{G}(d))$ and a 2-switch $\tau={{a \ b}\choose{c \ d}}$ on $H$ such that $\omega(\tau(H))-\omega(H)\geq 2$. Hence, $\alpha(\overline{\tau(H)})-\alpha(\overline{H})\geq 2$. Now, observe that $\overline{\tau(H)}=\tau^{-1}(\overline{H})$, where $\tau^{-1}={{a \ c}\choose{b \ d}}$. But, this means that
	\[ \alpha(\tau^{-1}(\overline{H}))-\alpha(\overline{H})\geq 2, \]
	which contradicts the stability of $\alpha$ in $\mathcal{G}(\overline{d})$ (\Cref{indep_stable}).
\end{proof}

\begin{theorem}
	Independence number, vertex-covering number and clique number have the interval property.
\end{theorem}

\begin{proof}
	It follows from \Cref{indep_stable,clique_stable}, and \Cref{vertex_cover_stable,stable_implies_int_prop}.
\end{proof}

\subsection{Domination number}\label{sectiondominationnumber}

A dominating set of a graph $G$ is a set $D$ of vertices such that every vertex of $G$ not in $D$ is adjacent to at least one element of $D$. Under this condition, we say that $D$ dominates (or covers) a vertex $v$, if $v$ is adjacent to some vertex of $D$ or if $v\in D$. A minimum dominating set is a dominating set of minimum size. The \emph{domination number} of $G$, denoted by $\gamma(G)$, is the size of a minimum dominating set of $G$.

\begin{theorem}
	\label{domin_stable}
	The domination number is stable.
\end{theorem}

\begin{proof}
	Let $G$ be a graph of order $\geq 4$, $D$ a minimum dominating set of $G$, and $\tau={{a \ b}\choose{c \ d}}$ a 2-switch on \(G\). If $D$ is a dominating set in $\tau(G)$, then $\gamma(\tau(G))\leq |D|=\gamma(G)\leq \gamma(G)+1$.

	Assume $D$ is not a dominating set in $\tau(G)$. As the edges incident to vertices not in $\{a,b,c,d\}$ in $G$ and in $\tau(G)$ are the same, $D$ dominates every vertex in $V(G)-\{a,b,c,d\}$. Hence at least one vertex in $\{a,b,c,d\}$ is not dominated by \(D\) in $\tau(G)$. Without loss of generality, assume that such a vertex is $a$ and consider its neighbors in \(G\) and \(\tau(G)\). Since the only edge incident to $a$ in $G$ that is not in $\tau(G)$ is $ab$, $b$ must be in $D$. Therefore, $d$ is dominated in $\tau(G)$ by $b$. Moreover, $D\cup a$ is a dominating set in $\tau(G)$, because $c$ is dominated by $a$ in $\tau(G)$. Thus, $\gamma(\tau(G))\leq |D\cup a|=\gamma(G)+1$.

	As in either case $\gamma(\tau(G))\leq \gamma(G)+1$, \Cref{lemastable} implies that $\gamma$ is stable.
\end{proof}

\begin{theorem}
	The domination number has the interval property.
\end{theorem}

\begin{proof}
	It follows from \Cref{domin_stable} and \Cref{stable_implies_int_prop}.
\end{proof}

\subsection{Number of connected components}\label{sectionnumberofcomponents}

Previously, we said that two forests with the same degree sequence $d$ must have the same number of connected components. We can rephrase this by saying that $\kappa$ is constant on $V(\mathcal{F}(d))$, and therefore, it (trivially) has the interval property in $\mathcal{F}(d)$. We will prove that the same occurs in $\mathcal{G}(d)$.

\begin{theorem}\label{kappa_stable}
The number of connected components is stable.
\end{theorem}

\begin{proof}
Let \(G\) be a graph, \(\tau={{a \ b}\choose{c \ d}}\) a 2-switch on \(G\) and $G'=G-\{ab,cd\}$. Clearly, $\kappa(G')\leq \kappa(G)+2$, as deleting an edge increases the number of components by at most 1. If $\kappa(G')\leq \kappa(G)+1$, then
\[ \kappa(\tau(G))=\kappa(G'+\{ac,bd\})\leq \kappa(G)+1, \]
as adding edges cannot increase the number of components. Assume $\kappa(G')=\kappa(G)+2$. If $a$ and $c$ are in different components of $G'$, then
\[ \kappa(G'+ac)=\kappa(G')-1\leq \kappa(G)+1, \]
which implies $\kappa(\tau(G))\leq \kappa(G)+1$. Suppose $a$ and $c$ are in the same component $H$ of $G'$. Then, neither $b$ nor $d$ are in $H$, because $\kappa(G')=\kappa(G)+2$. Furthermore, $b$ and $d$ must be in different components, since otherwise $\kappa(G)=\kappa(G'+\{ab,cd\})=\kappa(G')-1$.
Hence,
\[ \kappa(G'+bd)=\kappa(G')-1\leq \kappa(G)+1, \]
implying that $\kappa(\tau(G))\leq \kappa(G)+1$. Therefore, $\kappa$ is stable by \Cref{lemastable}.
\end{proof}

\begin{theorem}
	The number of connected components has the interval property.
\end{theorem}

\begin{proof}
	It follows from \Cref{kappa_stable} and \Cref{stable_implies_int_prop}.
\end{proof}

\subsection{Path-covering number and related parameters}\label{sectionpathcover}

Let $G$ be a graph. Two paths in $G$ that do not share vertices are said to be vertex-disjoint. A \emph{path covering} of $G$ is a set of vertex-disjoint paths of $G$ containing all the vertices of $G$. The \emph{path-covering number} of $G$, denoted by $\pi(G)$, is the minimum number of paths in a path-covering of $G$. A minimum path-covering in $G$ is a path-covering in $G$ of size $\pi(G)$. We can look at a path-covering $\Phi$ of $G$ as a spanning forest of $G$, whose components are just the paths of $\Phi$.

\begin{theorem}\label{Teo_Stability_PathCovering}
The path-covering number is stable.
\end{theorem}

\begin{proof}
	Let \(G\) be a graph and $\tau={{a \ b}\choose{c \ d}}$ a 2-switch on \(G\). Suppose that $\Phi$ is a minimum path-covering in $G$. Then, $\kappa(\Phi)=\pi(G)$. There are three cases:
	\begin{enumerate}
		\item \(ab\) and \(cd\) are both edges of the forest $\Phi$.
		\item \(ab\) and \(cd\) are not edges of the forest $\Phi$.
		\item exactly one of \(ab\) and \(cd\) is an edge of the forest $\Phi$.
	\end{enumerate}
	\begin{enumerate}[label=(\arabic*).]
		\item If \(\tau(\Phi)\) is a forest, then \(\tau(\Phi)\) is a path-covering of \(\tau(G)\). Hence, \(\pi(\tau(G))\leq \pi(G)\). If \(\tau(\Phi)\) is not a forest, then the vertices \(a,b,c,\) and \(d\) are all in the same path \(P\) of \(\Phi\). This path breaks into a cycle \(C\) and a path after applying \(\tau\). If \(e\) is an edge of \(C\), then \(\tau(\Phi)-e\) is a path-covering of \(\tau(G)\). Hence, \(\pi(\tau(G))\leq \pi(G)+1\).

		\item As \(\tau\) has no effect on \(\Phi\), this is still a path-covering in \(\tau(G)\). Hence, \(\pi(\tau(G))\leq \pi(G)\).

		\item Assume that \(ab \in E(\Phi)\) and \(cd \notin E(\Phi)\). Note that \(\tau(\Phi)-ab\) is a path-covering of \(\tau(G)\). Hence, \(\pi(\tau(G))\leq \pi(G)+1\).
	\end{enumerate}

	Therefore, \(\pi(\tau(G))\leq \pi(G)+1\), and so $\pi$ is stable by \Cref{lemastable}.
\end{proof}

Let $G$ be a graph. A subset $S$ of initially infected vertices of $G$ is called a forcing set if we can infect the entire graph by iteratively applying the following process. At each step, any infected vertex that has a unique uninfected neighbor infects this neighbor. The \emph{zero-forcing number} of $G$, which is denoted by \(Z(G)\), is the minimum size of a forcing set in $G$ (see~\cite{HogbenLinShader2022} for a comprehensive treatment). It is well-known that for all graphs \(G\), \(\pi(G) \leq Z(G)\), and that for trees the path-covering number agrees with the zero-forcing number (see \cite{work2008zero}, Proposition 4.2). Therefore, we have the following.

\begin{corollary}
	\label{zero_forc_stable}
	The zero-forcing number is stable in $\mathcal{F}(d)$.
\end{corollary}

\begin{proof}
	It follows from the previous discussion and \Cref{Teo_Stability_PathCovering}.
\end{proof}

Let \(G\) be a graph without isolated vertices. A sequence \((v_{1},\dots,v_{k})\), where \(v_{i}\in V(G)\), is called a Z-sequence if, for each \(i\): $N_i - \bigcup_{j=1}^{i-1} \left(N_j\cup\{v_j\}\right) \neq \varnothing$, where $N_i$ is the (open) neighborhood of $v_i$ in $G$. The \emph{Z-Grundy domination number} \(\gamma_{gr}^{Z}(G)\) of the graph \(G\) is the length of a longest Z-sequence. In \cite{bresar2017grundy} the following theorem was proved.

\begin{theorem}[\cite{bresar2017grundy}, Theorem 2.2]
	\label{zgrundydom_teo_citado}
	If \(G\) is a graph without isolated vertices, then $\gamma_{gr}^{Z}(G)+Z(G)=|V(G)|$.
\end{theorem}

\begin{corollary}
	\label{zgrundy_stable}
	Let $d$ be a degree sequence of a forest without isolated vertices. Then, the Z-Grundy domination number is stable in $\mathcal{F}(d)$.
\end{corollary}

\begin{proof}
	It follows from \Cref{zgrundydom_teo_citado} and \Cref{zero_forc_stable}.
\end{proof}

\begin{theorem}
	The path-covering number has the interval property. The zero-forcing number and the Z-Grundy domination number have the interval property in $\mathcal{F}(d)$.
\end{theorem}

\begin{proof}
	It follows from \Cref{Teo_Stability_PathCovering}, and \Cref{zero_forc_stable,zgrundy_stable,stable_implies_int_prop}.
\end{proof}

\subsection{Chromatic number}\label{sectionchromatic}

A coloring of $G$ is a function $f:V(G)\rightarrow \mathbb{N}$ such that $f(x)\neq f(y)$, for every $xy\in E(G)$. In this context, $f(V(G))$ is called the set of colors of $f$. If $|f(V(G))|=k$, we say that $f$ is a $k$-coloring of $G$. The \emph{chromatic number} of a graph $G$, denoted by $\chi(G)$, is the smallest value of $k$ for which there is a $k$-coloring of $G$.

\begin{theorem}
	\label{chrom_stable}
The chromatic number is stable.
\end{theorem}

\begin{proof}
	 Let $\tau={{a \ b}\choose{c \ d}}$ be a 2-switch on a graph $G$ which has a $k$-coloring $f$, with $k=\chi(G)$.
	 Notice that if $f$ is not a coloring of $\tau(G)$, then either $f(a)=f(c)$ or $f(b)=f(d)$. Assume without loss of generality that $f(V(G))=\{1,\ldots, k\}$. Define $\varphi:V(G)\rightarrow \mathbb{N}$ as
	 \[
\varphi(v)=\left\lbrace\begin{array}{ll}
f(v),\quad&\text{if }v\not\in\{c,d\},\\
k+1,\quad&\text{otherwise.}
\end{array}\right.
\]
Then $\varphi(a)\neq \varphi(c)$ and $\varphi(b)\neq \varphi(d)$, which implies that $\varphi$ is a coloring of $\tau(G)$.
Furthermore, $|\varphi(V(G))|=k+1$.
Therefore, $\chi(\tau(G)) \leq \chi(G)+1$, and so $\chi$ is stable by \Cref{lemastable}.
\end{proof}

\begin{theorem}
	The chromatic number has the interval property.
\end{theorem}

\begin{proof}
	It follows from \Cref{chrom_stable} and \Cref{stable_implies_int_prop}.
\end{proof}

\section*{Acknowledgements}
This work was partially supported by Universidad Nacional de San Luis, grants PROICO 03-0723 and PROIPRO 03-2923, MATH AmSud, grant 22-MATH-02, Consejo Nacional de Investigaciones
Cient\'ificas y T\'ecnicas grant, PIP 11220220100068CO and Agencia I+D+I grants PICT 2020-00549 and PICT 2020-04064.

We would like to thank the anonymous referee for their suggestions that greatly improved the presentation of this article. We also thank P\'eter L. Erd\H{o}s (A.\ R\'enyi Institute of Mathematics, Budapest) for bringing to our attention the early history of the switch operation and several relevant references.

\begin{Contacts}
\AuthorAddress%
  {Victor N. Schv\"ollner}
  {Instituto de Matem\'atica Aplicada San Luis, UNSL--CONICET, San Luis, Argentina}
  {vnsi9m6@gmail.com}
\AuthorAddress%
  {Adri\'an Pastine}
  {Instituto de Matem\'atica Aplicada San Luis, UNSL--CONICET, San Luis, Argentina}
  {agpastine@unsl.edu.ar}
\AuthorAddress%
  {Daniel A. Jaume}
  {Instituto de Matem\'atica Aplicada San Luis, UNSL--CONICET, San Luis, Argentina}
  {djaume@unsl.edu.ar}
\end{Contacts}

\end{document}